\documentclass[reqno]{amsart}

\usepackage{amsmath,amsthm,amssymb,amscd}
\usepackage{tabularx}
\usepackage{enumitem}
\newtheorem{thm}{Theorem}[section]

\newtheorem{conj}[thm]{Conjecture}
\newtheorem{lem}[thm]{Lemma}

\theoremstyle{definition}

\theoremstyle{remark}
\newtheorem{rem}[thm]{Remark}
\newtheorem{ex}[thm]{Example}
\numberwithin{equation}{section}

\begin{document}
	\title[On the $p$-adic valuation of third order linear recurrence sequences]
	{On the $p$-adic valuation of third order linear recurrence sequences}
	
	\author[Deepa Antony]{Deepa Antony (ORCID: 0000-0002-4214-2889)}
	
	\address{
		Department of Mathematics \\
		Indian Institute of Technology Guwahati \\
		Assam, India, PIN- 781039}
	
	\email{deepa172123009@iitg.ac.in}
	
	\author[Rupam Barman]{Rupam Barman (ORCID: 0000-0002-4480-1788)}
	\address{Department of Mathematics \\
		Indian Institute of Technology Guwahati \\
		Assam, India, PIN- 781039}
	\email{rupam@iitg.ac.in}
	
	\date{October 16, 2024}

	
	
	\begin{abstract}
	In a recent paper, Bilu et al. studied a conjecture of Marques and Lengyel on the $p$-adic valuation of the Tribonacci sequence. In this article, we study the $p$-adic valuation of third order linear recurrence sequences by considering a generalisation of the conjecture of Marques and Lengyel for third order linear recurrence sequences. Suppose that $(x_n)$ is a  third order linear recurrence sequence whose characteristic polynomial has a root $\gamma$ such that $|\gamma|>1$. We show that if there exists a prime $p$ for which the conjecture holds for $(x_n)$, then the  solution set of the Diophantine equation given by $x_n=m!$ in positive integers $n,m$ is finite. We also show that the solutions can be effectively computed when the form of the  conjecture is explicitly known. 
	\end{abstract}
	
	\maketitle
	\section{Introduction and statement of results} 
	For a prime $p$ and a non-zero integer $n$, the $p$-adic valuation of $n$, denoted $\nu_p(n)$, is defined as the highest power of $p$ which divides $n$. We set $\nu_p(0):=\infty$. It is an interesting problem to calculate $p$-adic valuation of integers having special properties, for example, linear recurrence sequences \cite{Bilu,Tripell, Lengyel,Len,Marques, Sanna1,Young}, Stirling numbers \cite{miska}, values assumed by functions \cite{medina} etc. Knowing the $p$-adic valuation is useful in solving certain Diophantine equations involving linear recurrence sequences and is also important in the investigation of denseness of such numbers in the $p$-adic numbers. In \cite{Lengyel}, Lengyel  completely characterized the $p$-adic valuation of the Fibonacci sequence. Later, Sanna \cite{Sanna1} solved the problem for any second order linear recurrence sequence with initial values $x_0=0$ and $x_1=1$. 
	\par In this article, we will  be interested in the $p$-adic valuation of linear recurrence sequences. It is an  interesting problem to characterise the $p$-adic valuation of a third order linear recurrence sequence. There are several generalizations of Fibonacci numbers. One of the most well-known is the Tribonacci sequence $(T_n)$, which is defined by the recurrence $T_{n+1}=T_{n}+T_{n-1}+T_{n-2}$, with initial values $T_0 = 0$ and $T_1 = T_2 = 1$. In \cite{Marques}, Marques and Lengyel calculated the $2$-adic valuation of the Tribonacci numbers $T_n$, and used their result to solve the Diophantine equation $T_n=m!$ in positive integers $n,m$. Later, the $2$-adic valuations of some other generalizations of Fibonacci numbers where calculated, see for example \cite{Len, sobo, Young}. Tripell sequences were also considered by Bravo et al. in \cite{Tripell}. 
	\par In \cite[Conjecture 8]{Marques},  Marques and Lengyel conjectured that the $p$-adic valuation of a Tribonacci number $T_n$ is either constant or is linearly dependent on the $p$-adic valuation of the index $n$ of the terms of the sequence for $n$ in some congruence class. However,  in \cite[Theorem 1.5]{Bilu} Bilu et al. showed that the conjecture fails for a specific infinite set of primes of relative density $1/12$. The aim of this article is to consider the conjecture for any third order linear recurrence sequence.  For any third order linear recurrence sequence $(x_n)$, we restate the conjecture as follows. We restate \cite[Conjecture 1.2]{Bilu} which is equivalent to \cite[Conjecture 8]{Marques}.
	\begin{conj}\label{con}\cite[Conjecture 1.2]{Bilu} Let $(x_n)$ be a third order linear recurrence sequence. Let $p$ be a prime number. There exists a positive integer $Q$ such that for every $i\in\{0,1,\dots,Q-1\}$ we have one of the following two options.
	\begin{enumerate}
		\item[\emph{(C)}] There exists $\kappa_i\in\mathbb{Z}_{\geq 0}$ such that for all but finitely many $n\in \mathbb{Z}$ satisfying $n\equiv i\pmod{Q}$ we have $\nu_p(x_n)=\kappa_i$.
		\item[\emph{(L)}] There exists $$a_i\in\mathbb{Z}, ~~~~\kappa_i\in\mathbb{Z}, ~~~~ \mu_i\in\mathbb{Z}_{>0}$$ 
		satisfying $$\nu_p(a_i-i)\geq \nu_p(Q),$$
		such that for all but finitely many $n\in \mathbb{Z}$ satisfying $n\equiv i\pmod{Q}$ we have $$\nu_p(x_n)=\kappa_i+ \mu_i \nu_p(n-a_i).$$
	\end{enumerate}
\end{conj} 
 For $a,b,c\in\mathbb{Z}$, and a prime $p$, we consider a third order linear recurrence sequence $(x_n)$ for $n\in\mathbb{Z}$ defined as 
\begin{align}\label{rec}
x_n=ax_{n-1}+bx_{n-2}+cx_{n-3} \text{ with  }x_0,x_1,x_2\in \mathbb{Z} \text{ and } p\nmid c.
\end{align}
The characteristic polynomial of the linear recurrence sequence \eqref{rec} is given by the polynomial $P(x)=x^3-ax^2-bx-c$. A linear recurrence sequence is said to be degenerate if its characteristic polynomial has a pair of distinct roots whose ratio is a root of unity. Otherwise it is said to be non-degenerate. Suppose that $(x_n)$ is a simple sequence, i.e, $P(x)$ has distinct roots. Let $\mathbb{K}=\mathbb{Q}_p(\lambda_1,\lambda_2,\lambda_3)$ be the splitting field of $P(x)$ over $\mathbb{Q}_p$, where $\lambda_1,\lambda_2,\lambda_3$ are the distinct roots of $P(x)$ and $\mathbb{Q}_p$ is the field of $p$-adic numbers. Let $\mathcal{O}$ be the  ring of integers of $\mathbb{K}$. Throughout the article, we consider primes $p$ which do not divide the discriminant of the characteristic polynomial so that $\mathbb{K}$ is unramified over $\mathbb{Q}_p$. 
\par 
For $n\geq 0$, the $n$th term of the sequence $(x_n)$ is given by the formula
\begin{align}\label{term}
	x_n=\sum_{i=1}^{3}c_{\lambda_i}\lambda_ i^n ,\text{ where }c_{\lambda_i}=q(\lambda_i)/P'(\lambda_i),
\end{align}
where $q(x):=x_0x^2+(x_1-x_0a)x+(x_2-x_1a-x_0b)$.
Since $p$ does not divide the discriminant of the polynomial, we have, $c_{\lambda_i}\in\mathcal{O}$. By our assumption, $p$ does not divide $c$ which implies that the roots $\{\lambda_1,\lambda_2,\lambda_3\}\subset \mathcal{O}^{\times}$.
\par 
In this article, we consider non-degenerate third order linear recurrence sequences $(x_n)$ given by \eqref{rec} whose $n$th term satisfies the equation
\begin{align}\label{equal}
x_n=\sum_{i=1}^{3}c_{\lambda_i}\lambda_ i^n  ~~~~\text{for all} ~~n\in\mathbb{Z}.
\end{align}
For $n\in\mathbb{Z}$, we say that $n$ is a zero of $(x_n)$ if $$\sum_{i=1}^{3}c_{\lambda_i}\lambda_ i^n=0.$$ 
Bilu et al. introduced the notion of a \emph{Twisted Integral Zero} (TIZ), denoted by  $\mathcal{Z}$, which is defined as the set of integers $n$ such that $$\sum_{i=1}^{3}\xi_i c_{\lambda_i}\lambda_ i^n=0$$ for some roots of unity $\xi_1,\xi_2,\xi_3$. A systematic study of Twisted Zeros of the Tribonacci sequence can be found in \cite{Bilu-1}.
\par For a positive integer $n$, let $N_{p^n}$ be the order of the subgroup $\left<\lambda_1,\lambda_2,\lambda_3\right>$ in the multiplicative group  $(\mathcal{O}/p^n)^{\times}$. For $\ell\in\{0,1,\dots,N_{p^n}-1\}$, consider the analytic function $f_\ell:\mathbb{Z}_p\rightarrow \mathbb{Z}_p$ defined by
\begin{align}\label{f}
f_\ell(x)=\sum_{i=1}^{3}c_{\lambda_i}\lambda_i^\ell\exp_p(x\log_p(\lambda_i^N))\text{, where }N=N_{p^n}.
\end{align}
Here, $\mathbb{Z}_p$ denotes the ring of $p$-adic integers in $\mathbb{Q}_p$.
 We have, $$f_\ell(m)=x_{\ell+mN}\text{ for all  }m\in\mathbb{Z}$$ and it can be seen that $N_{p^n}$ is the period of the sequence modulo $p^n$. 
\par 

It can be verified that the proof of the set of equivalent conditions of Conjecture 1.2 in \cite{Bilu} in terms of the zeros of the functions $f_\ell(x)$ holds true in the case of a third order linear recurrence sequence \eqref{rec} satisfying condition \eqref{equal} as well. The equivalent conditions are as follows. 
\begin{thm}\label{cod}\cite[Theorem 6.1]{Bilu} The following three statements are equivalent.
	\begin{enumerate}
		\item Conjecture \ref{con} holds for the given $p$.
		\item For every $\ell\in \{0,\dots,N-1\}$, the zeros of the function $f_\ell(z)$ belong to $N^{-1}\mathbb{Z}$.
		\item For every $\ell$ the following holds: if $b\in \mathbb{Z}_p$ is a zero of $f_\ell(z)$ then $\ell+Nb\in\mathcal{Z}$.
	\end{enumerate}
\end{thm}
Note that the equivalent statements hold for any $N=N_{p^n}$.
\par Using the equivalent conditions, we obtain two results that can be used to find primes $p$ for which the conjecture is true or false for sequences of the form \eqref{rec}. 
	 \begin{thm}\label{T1}
		Let $x_n=ax_{n-1}+bx_{n-2}+cx_{n-3}$ be a third order linear recurrence sequence with initial values $x_0,x_1,x_2$ satisfying equation \eqref{equal}. Suppose that the characteristic polynomial $P(x):=x^3-ax^2-bx-c$ of $(x_n)$ is irreducible over $\mathbb{Q}$ and $\lambda_1,\lambda_2,\lambda_3$ are the roots of $P(x)$ in some extension of $\mathbb{Q}_p$. Also, assume that $p$ divides neither  $c$ nor the discriminant of $P(x)$. Then the following holds:
		\begin{enumerate}
			\item If $p=2$, then if there exists an $\ell \in \{0,1,\dots,N_{p^2}-1\}$ such that $p^2| x_\ell$, $x_{\ell+N_{p^2}}-x_\ell\not\equiv 0\pmod{p^3}$ and  $$\ell-\frac{x_\ell}{p^2}\left(\frac{x_{\ell+N_{p^2}}-x_\ell}{p^2}\right)^{-1}N_{p^2}\not\equiv r\pmod{p}$$for all $r\in\mathcal{Z}$, then  Conjecture \ref{con} does not hold for $p$.
			\item If $p\neq 2$, then if there exists an $\ell \in \{0,1,\dots,N_{p}-1\}$ such that $p| x_\ell$, $x_{\ell+N_{p}}-x_\ell\not\equiv 0\pmod{p^2}$ and  $$\ell-\frac{x_\ell}{p}\left(\frac{x_{\ell+N_{p}}-x_\ell}{p}\right)^{-1}N_{p}\not\equiv r\pmod{p}$$for all $r\in\mathcal{Z}$, then  Conjecture \ref{con} does not hold for $p$.
		\end{enumerate}
	\end{thm}
 \begin{thm}\label{T2}
	Let $x_n=ax_{n-1}+bx_{n-2}+cx_{n-3}$ be a third order linear recurrence sequence with initial values $x_0,x_1,x_2$ satisfying equation \eqref{equal}. Suppose that the characteristic polynomial $P(x):=x^3-ax^2-bx-c$ of $(x_n)$ is irreducible over $\mathbb{Q}$ and $\lambda_1,\lambda_2,\lambda_3$ are the roots of $P(x)$ in some extension of $\mathbb{Q}_p$. Also, assume that $p$ divides neither $c$ nor the discriminant of $P(x)$. Then the following holds:
	\begin{enumerate}
		\item If $p=2$, then if for all $\ell \in \{0,1,\dots,N_{p^2}-1\}$, $p^2| x_\ell$ implies $x_{\ell+N_{p^2}}-x_\ell\not\equiv 0\pmod{p^3}$ then, Conjecture \ref{con}  holds for $p$ if $\ell\equiv r\pmod{N_{p^2}}$ for some $r\in\mathcal{Z}$.
		\item If $p\neq 2$, then if for all $\ell \in \{0,1,\dots,N_{p}-1\}$, $p| x_\ell$ implies  $x_{\ell+N_p}-x_\ell\not\equiv 0\pmod{p^2}$ then, Conjecture \ref{con}  holds for $p$ if $\ell\equiv r\pmod{N_{p}}$ for some $r\in\mathcal{Z}$.
	\end{enumerate}
\end{thm}
As evident from the previous theorems, it is beneficial to have a way to find out the Twisted Integral Zeros of sequences. The result below gives certain conditions which ensure that for certain linear recurrence sequences, zeros are the only possible Twisted Integral Zeros. Using the Skolem tool \cite{skolem}, zeros of linear recurrence sequences can be easily calculated. 
\begin{thm}\label{T}
	Let $x_n=ax_{n-1}+bx_{n-2}+cx_{n-3}$ be a third order linear recurrence sequence with initial values $x_0,x_1,x_2$ satisfying equation \eqref{equal}. Suppose that the characteristic polynomial $P(x):=x^3-ax^2-bx-c$ of $(x_n)$ is irreducible over $\mathbb{Q}$ and $\lambda_1,\lambda_2,\lambda_3$ are the roots of $P(x)$ in $\mathbb{C}$. Then, the set of Twisted Integral Zeros of $(x_n)$ is equal to the set of zeros of $(x_n)$ if the following three conditions hold:
	\begin{enumerate}
		\item $\mathbb{Q}(\lambda_i)$ is not Galois for some $1\leq i\leq3$.
		\item For the $i$ satisfying condition (1), $\log{\bigl\lvert\frac{c(\lambda_i)}{c(\lambda_j)}\bigr\rvert}/\log{\bigl\lvert\frac{\lambda_j}{\lambda_i}\bigr\rvert}\notin\mathbb{Z}$ for some $1\leq j \leq 3,j\neq i$. 
		\item $\mathbb{Q}(\lambda_1,\lambda_2,\lambda_3)$ does not contain primitive cubic roots of unity.
	\end{enumerate} 
\end{thm}
We give some examples as applications of Theorem \ref{T}.
\begin{ex}\label{ex1}
	The Tribonacci sequence is defined as 
	$$x_n=x_{n-1}+x_{n-2}+x_{n-3}, x_0=0,x_1=x_2=1.$$ 
	In \cite{Bilu}, Bilu et al. proved that the set of Twisted Integral Zeros of $(x_n)$ is equal to $\{0,-1,-4,-17\}$ which is also the set of zeros of the sequence.
\end{ex}
\begin{ex}\label{ex2}
	The Tripell sequence is defined as $$x_n=2x_{n-1}+x_{n-2}+x_{n-3}, x_0=0,x_1=1,x_2=2.$$
	Let $\lambda_1$ be the real root of the characteristic polynomial $P(x)=x^3-2x^2-x-1$. We have, $\mathbb{Q}(\lambda_1)$ is not Galois. Here $q(x)=x$. The condition $(2)$ of Theorem \ref{T} reduces to 
	$$\log{\lvert\frac{q(\lambda_i)P'(\lambda_j)}{q(\lambda_j)P'(\lambda_i)}\rvert}/\log{\lvert\frac{\lambda_j}{\lambda_i}\rvert}\approx 0.668 \notin\mathbb{Z}.$$
	Now, to show that $\mathbb{Q}(\lambda_1,\lambda_2,\lambda_3)$ does not contain primitive cubic roots of unity, observe that  $P(x)=x^3-2x^2-x-1\equiv(x+25)(x+27)(x+28)\pmod{41}$. Therefore, by Hensel's lemma, $P(x)$ has three distinct roots, say, $\lambda_1',\lambda_2',\lambda_3'$ in $\mathbb{Q}_{41}$ i.e, $\mathbb{Q}(\lambda_1',\lambda_2',\lambda_3')\subset\mathbb{Q}_{41}$ but $x^3-1$ does not split modulo $41$. Hence, $\mathbb{Q}_{41}$ does not contain primitive cubic roots of unity. Therefore, $\mathbb{Q}(\lambda_1',\lambda_2',\lambda_3')$ does not contain primitive cubic roots of unity. Since $\mathbb{Q}(\lambda_1,\lambda_2,\lambda_3)$ is isomorphic to $\mathbb{Q}(\lambda_1',\lambda_2',\lambda_3')$, $\mathbb{Q}(\lambda_1,\lambda_2,\lambda_3)$ does not contain primitive cubic roots of unity. Also, the Tripell sequence satisfies the equation \eqref{equal}. Hence, the set of Twisted Integral Zeros of $(x_n)$ is equal to the set of zeros of $(x_n)=\{-1,0\}$. The set of zeros is calculated by using Skolem tool \cite{skolem}.
\end{ex}
\begin{ex}\label{ex3}
	A slightly modified Tripell sequence is defined as $$x_n=2x_{n-1}+2x_{n-2}+x_{n-3}, x_0=0,x_1=1,x_2=2.$$
	Let $\lambda_1$ be the real root of the characteristic polynomial $P(x)=x^3-2x^2-2x-1$. We have, $\mathbb{Q}(\lambda_1)$ is not Galois. Here $q(x)=x$. The condition $(2)$ of Theorem \ref{T} reduces to 
	$$\log{\lvert\frac{q(\lambda_i)P'(\lambda_j)}{q(\lambda_j)P'(\lambda_i)}\rvert}/\log{\lvert\frac{\lambda_j}{\lambda_i}\rvert}\approx 0.861 \notin\mathbb{Z}.$$ Since $P(x)\equiv(x+19)(x+30)(x+31)\pmod{41}$, following the same argument as given in Example \ref{ex2}, we find that $\mathbb{Q}(\lambda_1,\lambda_2,\lambda_3)$ does not contain primitive cubic roots of unity. Hence, by Theorem \ref{T}, the set of Twisted Integral Zeros of $(x_n)$ is equal to the set of zeros of $(x_n)=\{-1,0\}$. We again calculate the zeros of $(x_n)$ by using Skolem tool \cite{skolem}.
\end{ex}
Using the above results, we run a program for primes upto $1000$ for the recurrence sequences defined in Examples \ref{ex2} and \ref{ex3} respectively, and obtain the following.
\begin{thm}\label{mtri} (1) For the Tripell sequence as defined in Example \ref{ex2}, Conjecture \ref{con} fails for $$p\in [5,1000]\backslash\{5,19,29,41,103,137,151,191,283,397,487,491,571,709,773,787,$$ $$877,883,971,983\}$$ and Conjecture \ref{con} holds for $$p\in\{103,137,191,397,487,491,709,773,787,883,971,983\}$$ in the form 
		\[\nu_p(x_n)= 
	\begin{cases}
	\nu_p(n+c)+1,& \text{if } n\equiv -c \pmod{Q_p},-c\in\{0,-1\};\\
		0,&     \text{ otherwise}.
	\end{cases}
	\] where $Q_p$ is given below.\\
	\begin{tabular}{|c|c|c|c|c|c|c|c|c|c|c|c| r |}
		\hline			
		$  p$ & $103$ & $137$ & $191$ & $397$ & $487$ & $491$ & $709$ & $773$ & $787$ & $883$ & $971$ & $983$ \\
		\hline
		$Q_p$ & $102$ & $136$ & $95$  & $198$ & $486$ & $245$ & $708$ & $772$ & $262$ & $882$ & $970$ & $491$ \\
		
		\hline  
	\end{tabular}\\\\
(2) For the modified Tripell sequence as defined in Example \ref{ex3}, Conjecture \ref{con} fails for $$p\in [5,1000]\backslash\{5,7,23,41,83,131,193,227,293,397,401,659,701,787,983\}$$ and	Conjecture \ref{con} holds for $$p\in \{5,23,41,131,193,227,293,401,659,701,787,983\}$$ in the form 
\[\nu_p(x_n)= 
\begin{cases}
	\nu_p(n+c)+1,& \text{if } n\equiv -c \pmod{Q_p},-c\in\{0,-1\};\\
	0,&     \text{ otherwise}.
\end{cases}
\] where $Q_p$ is given below.\\
	\begin{tabular}{|c|c|c|c|c|c|c|c|c|c|c|c| r |}
	\hline			
	$  p$ & $5$ & $23$ & $41$ & $131$ & $193$ & $227$ & $293$ & $401$ & $659$ & $701$ & $787$ & $983$ \\
	\hline
	$Q_p$ & $8$ & $22$ & $40$  & $130$ & $192$ & $113$ & $292$ & $400$ & $658$ & $350$ & $786$ & $982$ \\
	
	\hline  
\end{tabular}\\
	 \end{thm}
Next, we calculate $2$-adic valuation of the modified Tripell sequence as defined in Example \ref{ex3} explicitly.
\begin{thm}\label{val}
	Let $(x_n)$ be the modified Tripell sequence. For $n\geq1$, we have
	\[\nu_2(x_n)= 
	\begin{cases}
		0,& \text{if } n\equiv 1,4 \pmod{6},\\
		1,&       \text{if } n\equiv 2,3 \pmod{6},       \\
		2+\nu_2(n),& \text{if } n\equiv 0 \pmod{6}, \\
		3+\nu_2(n+1),& \text{if } n\equiv 5 \pmod{6}.
	\end{cases}
	\]
\end{thm}
We use the $2$-adic valuation to find the solutions of a Diophantine equation and obtain the following result. 
\begin{thm}\label{d1}
	For the linear recurrence sequence defined as  $x_n=2x_{n-1}+2x_{n-2}+x_{n-3},x_0=0,x_1=1,x_2=2,$ the only solutions of the Diophantine equation $x_n=m!$ in positive integers $n,m$ are $(n,m)\in \{(1,1),(2,2),(3,3)\}$.
\end{thm}
In \cite[Theorem 2]{sobo}, Sobolewski examined the Diophantine equation $\prod_{i=1}^{d}x_{n_i}=m!$ for a general linear recurrence sequence $(x_n)$. He observed some properties for $(x_n)$ which guarantees that the equation has only finitely many solutions in positive integers $m,n_1,\dots,n_d$. In the following theorem, we demonstrate that the Diophantine equation $x_n=m!$, with $(x_n)$ satisfying a certain third order linear recurrence, has only finitely many solutions if there exist a prime $p$ for which Conjecture \ref{con} holds. 
 \begin{thm}\label{d2}
	Let $(x_n)$ be a  linear recurrence sequence defined by $x_n=ax_{n-1}+bx_{n-2}+cx_{n-3}$, with initial values $x_0,x_1,x_2$ not all zeroes such that the characteristic polynomial has a root $\gamma$ with $|\gamma|>1$. If there exists a prime $p$ for which Conjecture \ref{con} holds for $(x_n)$, then the Diophantine equation $x_n=m!$ has finitely many solutions in positive integers $(n,m)$ and the solutions can be effectively computed when the form of Conjecture \ref{con} is explicitly known.
\end{thm}
	\section{Preliminaries}
Let $r$ be a non-zero rational number. Given a prime number $p$, $r$ has a unique representation of the form $r= \pm p^k a/b$, where $k\in \mathbb{Z}, a, b \in \mathbb{N}$ and $\gcd(a,p)= \gcd(p,b)=\gcd(a,b)=1$. The $p$-adic valuation of $r$ is defined as $\nu_p(r)=k$ and its $p$-adic absolute value is defined as $|r|_p=p^{-k}$. By convention, $\nu_p(0)=\infty$ and $|0|_p=0$. The $p$-adic metric on $\mathbb{Q}$ is $d(x,y)=|x-y|_p$. 
The field $\mathbb{Q}_p$ of $p$-adic numbers is the completion of $\mathbb{Q}$ with respect to the $p$-adic metric. We denote by $\mathbb{Z}_p$ the ring of $p$-adic integers which is the set of elements of $\mathbb{Q}_p$ with $p$-adic absolute value less than or equal to $1$.
The $p$-adic absolute value can be extended to a finite normal extension $\mathbb{K}$ over $\mathbb{Q}_p$ of degree  $n$. 
For $\alpha\in \mathbb{K}$, consider the $\mathbb{Q}_p$ linear map from $\mathbb{K}$ to $\mathbb{K}$ defined as $T_{\alpha}(x)=\alpha x$. The $p$-adic absolute value of $\alpha$, $|\alpha|_p$ is defined as the $n$th root of the determinant of the matrix representation of  $T_{\alpha}$ over $\mathbb{Q}_p$. 
Also, $\nu_p(\alpha)$ is the unique rational number satisfying $|\alpha|_p=p^{-\nu_p(\alpha)}$.
\par 
The ring of integers of $\mathbb{K}$, denoted by $\mathcal{O}$, is defined as the set of all elements in $\mathbb{K}$ with $p$-adic absolute value less than or equal to one. The ring of integers of $\mathbb{Q}_p$ is denoted by $\mathbb{Z}_p^\times$. A function $f: \mathcal{O}\rightarrow \mathcal{O}$ is called analytic if there exists a sequence $(a_n)$ in $\mathcal{O}$ such that 
$$ f(x)=\sum_{n=0}^{\infty}a_nx^n$$ for all $x\in \mathcal{O}$.
\par We recall  definitions of $p$-adic exponential and logarithmic function.  For $a\in\mathbb{K}$ and $r>0$, we denote $\mathcal{D}(a,r):=\{z\in \mathbb{K}: |z-a|_p<r\}$. Let $\rho=p^{-1/(p-1)}$.
\par 
For $z\in \mathcal{D}(0,\rho)$, the $p$-adic exponential function is defined as $$\exp_p(z)=\sum_{n=0}^{\infty}\frac{z^n}{n!}.$$ The derivative is given by $\exp'(z)=\exp(z)$.
For $\mathcal{D}(1,1)$, the $p$-adic logarithmic function is defined as  $$\log_p(z)=\sum_{n=1}^{\infty}\frac{(-1)^{n-1}(z-1)^n}{n}.$$ We have $\log_p'(z)=\frac{1}{z}$.
For $z\in \mathcal{D}(1,\rho)$, we have $\exp_p(\log_p(z))=z$. If $\mathbb{K}$ is unramified and $p\neq 2$, then $ \mathcal{D}(0,\rho)= \mathcal{D}(0,1)$ and  $\mathcal{D}(1,\rho)= \mathcal{D}(1,1)$.
	\par Next, we state two results for analytic functions which will be used in the proofs of our theorems.
	\begin{thm}\cite[Hensel's lemma]{gouvea} Let $f: \mathcal{O}\rightarrow \mathcal{O}$ be analytic.
	Let $b_0\in \mathcal{O}$ be such that $|f(b_0)|_p<1$ and $|f'(b_0)|_p=1$. Then there exists a unique $b\in\mathcal{O}$ such that $f(b)=0$ and $|b-b_0|_p<|f(b_0)|_p$. 
	\end{thm}
\begin{thm}\cite[Strassman's Theorem]{gouvea}
	Let $f: \mathcal{O}\rightarrow \mathcal{O}$ be analytic. Assume that $f(z)$ does not vanish identically on $\mathcal{O}$; equivalently, the coefficients $a_0,a_1,\dots$ are not all $0$. Define $\mu$ as the largest $m$ with the property $$ |a_m|_p=\max\{|a_n|_p:n=0,1,\dots\}.$$ Then $f(z)$ has at most $\mu$ zeros on $\mathcal{O}$.
\end{thm}
The following lemma is due to Bilu et al.
\begin{lem}\cite[Lemma 2.5]{Bilu}\label{l2}
	Let $\alpha$ be an algebraic number of degree $3$. Assume that $\mathbb{Q}(\alpha)$ is not a Galois extension of $\mathbb{Q}$. Let $\alpha_1(=\alpha),\alpha_2,\alpha_3$ be the conjugates of $\alpha$ over $\mathbb{Q}$. Assume further that the field $\mathbb{Q}(\alpha_1,\alpha_2,\alpha_3)$ does not contain primitive cubic roots of unity. Let $\xi_1,\xi_2,\xi_3$ be roots of unity such that
	$$\alpha_1\xi_1+\alpha_2\xi_2+\alpha_3\xi_3=0.$$
	Then $\xi_1=\xi_2=\xi_3$ and hence $\alpha_1+\alpha_2+\alpha_3=0$. 
\end{lem}
The following lemma gives bounds for the $p$-adic valuation of $m!$ for any positive integer $m$.
\begin{lem}\cite[Lemma 2.1]{Tripell}\label{l1}
	For any integer $m\geq 1$ and prime $p$, we have 
	$$\frac{m}{p-1} -\left\lfloor{ \frac{\log m}{\log p}}\right\rfloor -1 \leq \nu_p(m!)\leq \frac{m-1}{p-1},$$
	where $\lfloor x\rfloor$ denotes the largest integer less than or equal to $x$.
\end{lem}
	\section{Proof of the theorems}
	In this section, we prove all our results. We first give a proof of Theorem \ref{T1}. 
	We prove a lemma which will be used in the proof of the theorem.
		\begin{lem}\label{lem-new}
			Let $(x_n)$ be a linear recurrence sequence defined by \eqref{rec} satisfying equation \eqref{equal}. Suppose that $N=N_{p^2}$. For $f_\ell$ defined by \eqref{f}, if $p^2\nmid x_\ell$, then $f_\ell(z)\neq 0$ for all $z\in \mathbb{Z}_p$.
	\end{lem}
	\begin{proof}
	Suppose that $p^2 \nmid x_ \ell$.  Since $x_n\equiv x_\ell \pmod{p^2}$ for $n \equiv \ell \pmod{N}$, we have $\nu_p(x_n)<2$ for all $n\equiv \ell\pmod{N}$. Therefore,  $|f_\ell(m)|_p>p^{-2}$ for all integers $m$. Hence, $|f_\ell(z)|_p \geq p^{-2}$ for all $z\in\mathbb{Z}_p$. This completes the proof.
	\end{proof}
Having Lemma \ref{lem-new} showed, we are ready to prove Theorem \ref{T1}.
	\begin{proof}[Proof of Theorem \ref{T1}]
	 Suppose that $p=2$, $N=N_{p^2}$ and $\ell\in\{0,\dots,N_{p^2}-1\}$. If $p^2\nmid x_\ell$, then by Lemma \ref{lem-new}, $f_\ell(x)$ has no zero in $\mathbb{Z}_p$. Hence, we consider $\ell$ such that 
 $p^2|x_\ell$. We define an analytic function $$g(x):=f_\ell(x)/p^2=\sum_{k=0}^{\infty}\beta_kx^k.$$
 Then, we have
	\begin{align*}\beta_0&=g(0)=\frac{f_\ell(0)}{p^2}=\frac{x_\ell}{p^2}\in \mathbb{Z},\\
	\beta_k&=\frac{p^{2(k-1)}}{k!}\sum_{i=1}^{3}c_{\lambda_i}\lambda_i^\ell \left(\frac{\log\lambda_i^N}{p^2}\right)^k
	\end{align*} for $k\geq1$. Note that $\nu_p(\beta_k)>0$ for $k\geq2$. Also,
	$$g'(x)=\beta_1+\sum_{k=2}^{\infty}k\beta_kx^{k-1}\equiv \beta_1 \pmod{p}.$$
Now,
	\begin{align*}g'(0)=\beta_1&=\sum_{i=1}^{3}c_{\lambda_i}\lambda_i^\ell\frac{\log\lambda_i^N}{p^2}\\
	&\equiv \sum_{i=1}^{3}c_{\lambda_i}\lambda_i^\ell\frac{\lambda_i^N-1}{p^2}\pmod{p}\\
	&\equiv \frac{x_{\ell+N}-x_{\ell}}{p^2}\pmod{p}.
	\end{align*}
 Therefore, $x_{\ell+N}-x_\ell\not\equiv 0\pmod{p^3}$ implies 	$g'(0)=\beta_1\not\equiv0\pmod{p}$.
 \par 
 For $b_0\equiv -\beta_0\beta_1^{-1}\pmod{p}$,
	we get, $$g(b_0)\equiv 0\pmod{p},g'(b_0)\equiv \beta_1 \not\equiv 0 \pmod{p}. $$ Therefore, by Strassman's Theorem and Hensel's Lemma, $g$ has a unique zero $b\equiv b_0\pmod{p}$ in $\mathbb{Z}_p$. Hence, $f_\ell(x)$ has a unique zero $b$  in $\mathbb{Z}_p$.
\par If $\ell+b_0N\not\equiv r\pmod{p}$ for all $r\in\mathcal{Z}$, then $\ell+bN\notin \mathcal{Z}$. But we have,  $$\ell+b_0N\equiv \ell-\frac{x_\ell}{p^2}\left(\frac{x_{\ell+N}-x_\ell}{p^2}\right)^{-1}N\not\equiv r\pmod{p} \text{ for all }r\in\mathcal{Z}.$$ Therefore, by Theorem \ref{cod}, the conjecture fails for $p$.
\par For $p\neq 2$, the proof proceeds along similar lines to the proof of \cite[Theorem 8.1]{Bilu}. 
		\end{proof}
\begin{proof}[Proof of Theorem \ref{T2}] The proof proceeds along similar lines to the proof of Theorem 8.2 in \cite[Theorem 8.2]{Bilu}, so we omit the details for reasons of brevity.
	\end{proof}
\begin{proof}[Proof of Theorem \ref{T}]
	We define $\alpha_i:=\frac{q(\lambda_i)}{P'(\lambda_i)}\lambda_i^n$. Then, an integer $n$ is a Twisted Integral Zero of $(x_n)$ if $\sum_{i=1}^{3}\xi_i\alpha_i=0$ for some roots of unity $\xi_1, \xi_2, \xi_3$.
	\par 
	 Suppose that $\mathbb{Q}(\lambda_i)$ is not Galois for some $1\leq i\leq 3$. Clearly, $\alpha_i\in\mathbb{Q}(\lambda_i)$. If $\mathbb{Q}(\alpha_i)\neq\mathbb{Q}(\lambda_i)$, then $\alpha_i\in\mathbb{Q}$ since $[\mathbb{Q}(\lambda_i):\mathbb{Q}]=3$. Since $\mathbb{Q}(\lambda_i)$ is not Galois, we have $\text{Gal}(\mathbb{Q}(\lambda_1,\lambda_2,\lambda_3)/\mathbb{Q})$ is isomorphic to the symmetric group $S_3$.  For $\sigma=(i,j)\in S_3$, we have  $\sigma(\alpha_i)=\alpha_j$. Since $\sigma$ fixes $\alpha_i\in\mathbb{Q}$, we have $$\frac{q(\lambda_i)}{P'(\lambda_i)}\lambda_i^n=\frac{q(\lambda_j)}{P'(\lambda_j)}\lambda_j^n$$ which yields $$n=\log{\lvert\frac{q(\lambda_i)P'(\lambda_j)}{q(\lambda_j)P'(\lambda_i)}\rvert}/\log{\lvert\frac{\lambda_j}{\lambda_i}\rvert}\notin\mathbb{Z},$$ and this gives a contradiction. Therefore, $\mathbb{Q}(\alpha_i)$ is equal to $\mathbb{Q}(\lambda_i)$ which shows that $\mathbb{Q}(\alpha_i)$ is not Galois. Also, since $\mathbb{Q}(\alpha_1,\alpha_2,\alpha_3)\subset\mathbb{Q}(\lambda_1,\lambda_2,\lambda_3)$, by condition $(3)$, $\mathbb{Q}(\alpha_1,\alpha_2,\alpha_3)$ does not contain primitive cubic roots of unity. Therefore, by Lemma \ref{l2}, we have $$\sum_{i=1}^{3}\xi_i\alpha_i=0 \text{ implies }\sum_{i=1}^{3}\alpha_i=0.$$ Hence, $n$ has to be a zero of $(x_n)$.
	\end{proof}
\begin{proof}[Proof of Theorem \ref{mtri}]
In view of Theorem \ref{T}, we obtain $\mathcal{Z}=\{-1,0\}$ for the Tripell and modified Tripell sequences. For both the sequences, we run a Python program for primes $p\leq 1000$ and search for $\ell$ which satisfy the conditions of Theorem \ref{T1}. This gives primes for which the conjecture fails. Similarly, by running a Python program for primes $p\leq 1000$, we check the conditions of Theorem \ref{T2} and obtain primes for which the conjecture is true. In the cases where the conjecture holds, the $p$-adic valuation can be calculated and the proof follows similar to the proof of Theorem 1.7 in  \cite[Theorem 1.7]{Bilu}.
		\end{proof}
	\begin{proof}[Proof of Theorem \ref{val}]
			The period of the sequence modulo $2^2$ is $N_{2^2}=6$. It can be calculated that $\nu_2(x_n)$ is constant $0$ for $n\equiv1,4$ modulo $6$ and is constant $1$ for $n\equiv2,3$ modulo $6$. Now, we consider $f_\ell(x)$ for $\ell=0$ and $-1$.\\ For $n\equiv \ell \pmod{6}$, $x_n=f_\ell(x)$ where $x=\frac{n-\ell}{6}$.
		Define $$g_\ell(x)=\frac{f_\ell(x)}{2^2}=\sum_{k=0}^{\infty}\beta_kx^k,$$
		where
		\begin{align}\label{e}
		 \beta_k=\frac{f_\ell^k(0)}{2^2k!}=\frac{\sum_{i=1}^{3}c_{\lambda_i}\lambda_i^\ell(\log(\lambda_i^{6}))^k}{2^2k!}=\frac{2^{2(k-1)}}{k!}\sum_{i=1}^{3}c_{\lambda_i}\lambda_i^\ell\left( \frac{\log(\lambda_i^{6})}{2^2}\right) ^k,
		 	\end{align}
 where $\sum_{i=1}^{3}c_{\lambda_i}\lambda_i^\ell\left( \frac{\log(\lambda_i^{6})}{2^2}\right) ^k\in \mathbb{Z}_2$. We have $\beta_0=0$ for $\ell=0,-1$ and
		\begin{align*}
			\beta_1&=\frac{f_\ell'(0)}{2^2}=\frac{\sum_{i=1}^{3}c_{\lambda_i}\lambda_i^\ell\log(\lambda_i^{6})}{2^2}\\
			&\equiv\frac{1}{2^2}\sum_{i=1}^{3}c_{\lambda_i}{\lambda_i}^\ell\left( \lambda_i^6-1-\frac{(\lambda_i^6-1)^2}{2}\right)  \pmod{2^4}\\
			&\equiv \frac{1}{2^2} \left(x_{\ell+6}-x_\ell - \frac{(x_{\ell+2\times 6}-2x_{\ell+6}+x_\ell)}{2}\right) \pmod{2^4}.
		\end{align*}
	It can be calculated that
		for $\ell=0$, $\beta_1\equiv 2\pmod{2^4}$. Hence, $\nu_2(\beta_1)=1$. Similarly,
		for $\ell=-1$, we have $\beta_1\equiv 4\pmod{2^4}$. Hence, $\nu_2(\beta_1)=2$.\\
		Now, 
		$$\beta_2\equiv \frac{1}{2^2} \left(x_{\ell+2\times 6}-2x_{\ell+6}+x_\ell\right)  \equiv 0 \pmod{2^3} $$
		for both $\ell=0$ and $\ell=-1$. Also, we have
		$$\beta_3\equiv \frac{1}{2^2\times 3!} \left(x_{\ell+3\times 6}-3x_{\ell+2\times 6}-x_\ell+3x_{\ell+6}\right) \equiv 0 \pmod{2^3}$$	for $\ell=-1$.
	Moreover, by equation \eqref{e}, we obtain $$\nu_2(\beta_k)\geq \nu_2\left(\frac{2^{2(k-1)}}{k!}\right)=2(k-1)-\nu_2(k!)>2(k-1)-k=k-2.$$ Therefore, $\nu_2(\beta_k)>1$ for $k\geq3$ and $\nu_2(\beta_k)>2$ for $k\geq4$.
		Hence, for $\ell=0$, we have
		\begin{align*}
			\nu_2(x_n)=\nu_2(f_0(x)) 
			&=\nu_2\left(2^2\left(\sum_{k\geq0}\beta_kx^k\right)\right)\\
			&=\nu_2(2^2\beta_1x)\\
			&=3+\nu_2\left(\frac{n}{6}\right)\\
			&=2+\nu_2(n).
		\end{align*} Similarly, $\nu_2(x_n)=3+\nu_2(n+1)$ for $\ell=-1$.                                                                                                                                                                                                                                                                                                                                                                                                                                                                                                                                                                                                                                                                                                                                                                                                                                                                                                                                                                                                                                                                                                                                                                                                                                                                                                                                                                                                                                                                                                                                                                                                                                                                                                                                                                                                                                                                                      
		This is completes the proof of the theorem.
		\end{proof}
	Next, we prove a lemma which will be used in the proof of Theorem \ref{d1}.
	\begin{lem}\label{root}
		Let $(x_n)$ be a  linear recurrence sequence defined by $x_n=ax_{n-1}+bx_{n-2}+cx_{n-3},a,b,c>0$ such that the characteristic polynomial has a real root $\gamma>1$. Suppose that the initial values of the sequence satisfy the condition $$\gamma^{-1}\leq x_1\leq \gamma^0 \leq x_2 \leq \gamma^1 \leq x_3 \leq \gamma^2.$$ Then we have, for $n\geq 1$,
		$$\gamma^{n-2}\leq x_n\leq \gamma^{n-1}.$$
	\end{lem}
	\begin{proof}
		We prove the lemma using induction on $n$. The result holds for $n=1,2,3$ by the hypothesis. Let $n\geq 4$. Suppose that $\gamma^{m-2}\leq x_m\leq \gamma^{m-1}$ holds for all $m$ where $3\leq m\leq n-1$. Then, 
		$$a\gamma^{m-2}+b\gamma^{m-3}+c\gamma^{m-4}\leq ax_{m}+bx_{m-1}+cx_{m-2}\leq a\gamma^{m-1}+b\gamma^{m-2}+c\gamma^{m-3}$$ which implies
		\begin{align}\label{new-eqn2}
		\gamma^{m-4}(a\gamma^2+b\gamma+c)\leq x_{m+1}\leq \gamma^{m-3}(a\gamma^2+b\gamma+c).
		\end{align}
		Since $\gamma$ is a root of the polynomial  $x^3-ax^2-bx-c$, so \eqref{new-eqn2} gives
		$$\gamma^{m-1}\leq x_{m+1}\leq \gamma^{m}.$$
		This completes the proof.
	\end{proof}
Having Lemma \ref{root} showed, we are ready to prove Theorem \ref{d1}.
	\begin{proof}[Proof of Theorem \ref{d1}]
	Suppose that $x_n=m!$ for some positive integers $n$ and $m$. For $m\leq 5$, it can be checked that the only solutions are  $(n,m)\in \{(1,1),(2,2),(3,3)\}$.
	By Theorem \ref{val}, we have $\nu_2(x_n)=\nu_2(m!)\leq 3+\max\{\nu_2(n),\nu_2(n+1)\}$.
	 By Lemma \ref{l1}, we have
	$$m-\left\lfloor{\frac{\log{m}}{\log{2}}}\right\rfloor-1\leq \nu_2(m!)\leq 3+ 2\max\{ \nu_2(n),\nu_2(n+1),$$
	which yields
	$$ \frac{m}{2}-\frac{\log{m}}{2\log{2}}-\frac{5}{2}\leq \max\{\nu_2(n),\nu_2(n+1)\}.$$
	Therefore,
	$$2^{\frac{m}{2}-\frac{\log{m}}{2\log{2}}-\frac{5}{2}}\leq 2^{ \max\{\nu_2(n),\nu_2(n+1)\}}\leq n+1,$$
	and hence
	\begin{align}\label{1'}
		\frac{m}{2}-\frac{\log{m}}{2\log{2}}-\frac{5}{2}\leq \frac{\log{(n+1)}}{\log{2}}.
	\end{align}
	By Lemma \ref{root}, since $\gamma\approx 2.83$ is a root which satisfies the conditions of Lemma \ref{root}, we have $2.83^{n-2}\leq x_n=m!<\left(\frac{m}{2}\right)^m, \text{ for } m>5$ which gives
	$$(n-2)\log(2.83)<m\log\frac{m}{2}.$$ Hence,
	$$n<\frac{m\log\frac{m}{2}}{\log 2.83}+2.$$ Inserting this bound in equation \eqref{1'}, we get
	$$ \frac{m}{2}-\frac{\log{m}}{2\log{2}}-\frac{5}{2}\leq \frac{\log{(\frac{m\log\frac{m}{2}}{\log 2.83}+2+1)}}{\log{2}}$$ which gives, $m\leq 21$ and hence $n<52$. By a computational search, it can be checked that  $\{(1,1),(2,2),(3,3)\}$ are the only solutions.
\end{proof}
We state a theorem which will be used in the proof of Theorem \ref{d2}.
\begin{thm}\label{ever}\cite[Theorem 2.4]{Everest}
	For any integer non-degenerate linear recurrence sequence $a$ of order $n$ with $r\leq 3$ dominating characteristic roots (roots of the characteristic polynomial with maximum absolute value), 
	$$|a(x)|\geq |\alpha_1|^xx^{-k(a)},x\geq c(a)$$ where $k(a)$ and $c(a)$ are effective constants depending on the sequence and $\alpha_1$ is a dominating characteristic root.
\end{thm}
	\begin{proof}[Proof of Theorem \ref{d2}]
	Suppose that $x_n=m!$ for some positive integers $n$ and $m$.
If Conjecture \ref{con} holds for $p$, there exists $Q$ such that for $i\in\{0,1,\dots,Q-1\}$,  either $\nu_p(x_n)=\kappa_i+\mu_i(\nu_p(n-a_i))$ i.e, (L) holds or $\nu_p(x_n)=\kappa_i$ i.e, (C) holds for $n\equiv i\pmod{Q}$ except for finitely many $n$. Suppose that $\kappa=\max_i\{|\kappa_i|\}$. Define $L=\mu\max_i\{\nu_p(n-a_i)\}$ where $\mu=\max_i\{\mu_i\}$ if the set of $i$ such that (L) holds is nonempty and $L=0$ otherwise. Then, $$\nu_p(x_n)=\nu_p(m!)\leq \kappa+L.$$ By Lemma \ref{l1},
$$\frac{m}{p-1} -\left\lfloor{ \frac{\log m}{\log p}}\right\rfloor -1 \leq \nu_p(m!)\leq\kappa+\mu\max_i\{\nu_p(n-a_i)\}, \text{ when }L\neq 0$$ and 
$$\frac{m}{p-1} -\left\lfloor{ \frac{\log m}{\log p}}\right\rfloor -1 \leq \nu_p(m!)\leq\kappa,\text{ when }L= 0.$$ 
For $L\neq 0$, we have 
$$\frac{m}{\mu(p-1)} -\left\lfloor{ \frac{\log m}{\log p}}\right\rfloor\frac{1}{\mu} -\frac{1+\kappa}{\mu} \leq \max_i\{\nu_p(n-a_i)\}.$$
Therefore,
$$p^{\frac{m}{\mu(p-1)} -\left\lfloor{ \frac{\log m}{\log p}}\right\rfloor\frac{1}{\mu} -\frac{1+\kappa}{\mu}}\leq p^{\max_i\{\nu_p(n-a_i)\} }\leq n+a,\text{ where } a=|\max_i\{-a_i\}|$$
which gives
\begin{align}\label{1}
	\frac{m}{\mu(p-1)} -\left\lfloor{ \frac{\log m}{\log p}}\right\rfloor\frac{1}{\mu} -\frac{1+\kappa}{\mu}\leq \frac{ \log (n+a)}{\log p},\text{ where } a=|\max_i\{-a_i\}|.
\end{align}
    By Theorem \ref{ever}, 
    \begin{align}\label{8}
    	|\gamma|^{n}n^{-k(x)}\leq |x_n| \text{ for } n>c(x)
    	\end{align}
    	 where $x$ denotes the sequence $(x_n)$ and $\gamma$ is a dominating characteristic root such that $|\gamma|>1$. From inequality \eqref{8}, we can calculate a bound for values of $n$ for which $x_n=m!$ for $m\leq5$. Suppose that  $m>5$.
We have $$\frac{|\gamma|^{n}}{n^{k(x)}}\leq |x_n|=x_n=m!<\left(\frac{m}{2}\right)^m$$  which gives
\begin{align}\label{2}
	n\log|\gamma|-k(x)\log n<m\log\frac{m}{2}. 
\end{align}
Combining \eqref{1} and \eqref{2}, we get an effective bound for $m$ and $n$. Similarly, it can be shown that $m$ and $n$ are bounded effectively in the case when $L=0$ as well.  This completes the proof.
	\end{proof}
\begin{rem}
	In \cite{Bilu}, Bilu et al. studied Conjecture \ref{con} for Tribonacci sequences, and they showed that the conjecture fails for a specific infinite set of primes of relative density $1/12$.
	Following the same approach, we now give an example of a set of third order linear recurrence sequences for which Conjecture \ref{con} does not hold.
	\end{rem}
	\begin{ex}\label{exx}
		Consider linear recurrence sequences defined as
		\begin{align*}
		x_{n}=3x_{n-1}+ax_{n-2}+x_{n-3},x_0=0,x_1=1,x_2=3,
		\end{align*}
		where $a\in \mathbb{Z}$ is such that the characterisitic polynomial $P(x)=x^3-3x^2-ax-1$ is irreducible over $\mathbb{Q}$ and has a root $\lambda\in\mathbb{C}$ such that $\frac{\lambda}{P'(\lambda)^3}\notin\mathbb{R}$.  It can be shown that Conjecture \ref{con} fails for an infinite set of primes with relative density $1/12$. 
	Suppose that $\Lambda=\{\lambda_1,\lambda_2,\lambda_3\}$ is the set of roots of $P$ in some extension of $\mathbb{Q}_p$ where $\lambda_2$ and $\lambda_3$ are complex conjugates.
	Now, consider primes $p\equiv 2\pmod{3}$ for which $\Lambda\subset\mathbb{Q}_p$ and $p$ does not divide the discriminant of $P(x)$.  By the Chebotarev density theorem and Dirichlet's theorem, we obtain that such primes have relative density $1\slash 12$. We have, $f_\ell(m)=x_{\ell+mN}$ for $m\in \mathbb{Z}_{\geq 0}$, where $N=N_p$  and $\ell\in\{0,1,\dots,N-1\}$. Consider the function
	\begin{align*} 
	F(x,y,z)&=x^3+y^3+z^3-3xyz\\
	&=(x+y+z)(x+\zeta y+\overline{\zeta}z)(x+\overline{\zeta}y+\zeta z),
	\end{align*} 
	where $\zeta$ and $\overline{\zeta}$ are primitive cube roots of unity.  Since $p\equiv 2\pmod{3}$, $\mathbb{Q}(\lambda_1,\lambda_2,\lambda_3)$ does not contain any primitive cube roots of unity. Moreover, every element of $\mathbb{Z}_p^\times$ has only a single cube root in $\mathbb{Z}_p$. Let $\alpha_i=c_i\lambda_i^{-2/3}$, where $\lambda_i^{1/3}$ is the cube root of $\lambda_i\in \mathbb{Z}_p$.
	Since $\lambda_1+\lambda_2+\lambda_3=3$, we have, $$F(\alpha_1,\alpha_2,\alpha_3)=\frac{-3+\lambda_1+\lambda_2+\lambda_3}{(\lambda_1-\lambda_2)^2(\lambda_2-\lambda_3)^2(\lambda_1-\lambda_3)^2}=0.$$ Therefore, one of the following holds.
	\begin{align}
		c_1\lambda_1^{-2/3}+c_2\lambda_2^{-2/3}+c_3\lambda_3^{-2/3}&=0, \\
		c_1\lambda_1^{-2/3}+\zeta c_2\lambda_2^{-2/3}+\overline{\zeta}c_3\lambda_3^{-2/3}&=0 \label{4},\\ 
		c_1\lambda_1^{-2/3}+\overline{\zeta} c_2\lambda_2^{-2/3}+\zeta c_3\lambda_3^{-2/3}&=0 \label{3}.
	\end{align}
	If \eqref{4} holds, then applying $\sigma\in \text{Gal}(\mathbb{Q}_p(\zeta)/\mathbb{Q}_p)$ where  $\sigma(\zeta)= \overline\zeta$, we obtain \eqref{3}. Hence, we get $(\zeta-\overline{\zeta})(c_2\lambda_2^{-2/3}-c_3\lambda_3^{-2/3})=0$, which yields
	$ \frac{\lambda_2}{P'(\lambda_2)^3}=\frac{\lambda_3}{P'(\lambda_3)^3}$. This is  not possible since $\lambda/ P'(\lambda)^3\notin \mathbb{R}$. Therefore, we have, 	$c_1\lambda_1^{-2/3}+c_2\lambda_2^{-2/3}+c_3\lambda_3^{-2/3}=0$.
	Now, let $\ell$ be such that $-2\slash 3\equiv \ell\pmod{N}$. Then, there exists  $b\in\mathbb{Z}_p\cap\mathbb{Q}$ such that $\frac{-2}{3}=\ell+bN$. Similar  to the proof of Theorem 1.5 in \cite{Bilu},   we can show that $f_\ell(b)=0$. Consider $n\in \mathbb{Z}_{\geq 0}$ such that $n\equiv -2/3 \pmod{p-1}$. Then, $n\equiv \ell \equiv -2/3 \pmod{N}$. Therefore, $n=\ell+Nm$ for $m\in \mathbb{Z}_{\geq 0}  $. Since $f_\ell(b)=0$, we have
	$$\nu_p(x(n))=\nu_p(f_{\ell}(m))=\nu_p(f_{\ell}(m)-f_{\ell}(b)).$$ By Proposition $3.1$ in \cite{Bilu}, we have 
	$$\nu_p(f_{\ell}(m)-f_{\ell}(b))\geq \nu_p(m-b)=\nu_p(n+2/3).$$
	 Therefore, for $n\in \mathbb{Z}_{\geq 0}$, $n\equiv -2/3\pmod{p-1}$, we have $\nu_p{(x_n)}\geq \nu_p{(n+\frac{2}{3})}$. Let $(n_k)$ be a sequence of positive integers satisfying $n_k\equiv -2/3 \pmod{(p-1)p^k}$. If Conjecture \ref{con} is true for $p$, then for some $i\in \{0,\dots,Q-1\}$, the residue class $i\pmod{Q}$ contains infinitely many $n_k$. Since $\nu_p(n_k+2/3)\rightarrow \infty$, we have $\nu_p(x(n))\rightarrow \infty$. Hence, for this $i$ we must have option $(L)$ of Conjecture \ref{con} i.e, $\nu_p(x_{n_k})=\kappa_i+ \mu_i\nu_p(n_k-a_i)$ for some $a_i\in \mathbb{Z}$. Therefore $\nu_p(n_k-a_i)\rightarrow \infty$. But we also have that $\nu_p(n_k+2/3)\rightarrow\infty$ which gives a contradiction since $a_i\neq -2/3$. 
		\end{ex}
		\par Note that, in Example \ref{exx}, for almost all positive integers $a$, $P(x)$ has 3 real roots, providing a method that mostly works for negative $a$. One could consider flipping the sign of $a$.
\begin{rem}
In \cite[Theorem 1.8]{Bilu}, it was shown that the conjecture can be partially, but not completely, rescued by allowing rational twisted zeros in addition to integral twisted zeros.   It would be interesting to explore the same for the Tripell and modified Tripell sequence. Also, in view of Theorem \ref{d2}, it would be interesting to know,   given any third order linear recurrence sequence, does there exist a prime $p$ for which the conjecture holds?
\end{rem}
\section{Acknowledgement}
We are grateful to the referee for going through the article and providing many helpful comments. We are extremely grateful to Professor Florian Luca for many helpful comments while preparing the article.
	
\end{document}